\documentclass[a4paper,12pt]{amsart}

\usepackage{epsfig}
\usepackage{amsthm,amsfonts}
\usepackage{amssymb,graphicx,color}
\usepackage[all]{xy}

\newtheorem{theorem}{Theorem}[section]
\newtheorem{lemma}[theorem]{Lemma}

\newtheorem{proposition}[theorem]{Proposition}
\newtheorem{definition}[theorem]{Definition}

\newcommand{\C}{\mathbb C}
\newcommand{\R}{\mathbb R}

\begin{document}

\title[Lipschitz Regular complex algebraic sets are smooth. 
 ]{Lipschitz Regular complex algebraic sets are smooth. 
}

\author{L. Birbrair}
\author{A. Fernandes}
\author{L\^e  D. T.}
\author{J. E. Sampaio}
\address{L. ~Birbrair, A. ~Fernandes and L\^e  D. T. - 
Departamento de Matem\'atica,Universidade Federal do Cear\'a Av.
Humberto Monte, s/n Campus do Pici - Bloco 914, 60455-760
Fortaleza-CE, Brazil} 
\email{birb@ufc.br}
 \email{alexandre.fernandes@ufc.br}
\email{ledt@ictp.it}
\address{J. E. Sampaio - Instituto Federal de Educa\c c\~ao, Ci\^encia e Tecnologia do Cear\'a
Estrada do A\c cude do cedro, KM 5 s/n 63900-000
 Quixad\'a - CE, Brazil}
\email{edsontdr@hotmail.com}

%\thanks{The authors were partially supported by CNPq-Brazil}

\begin{abstract}
A classical Theorem of Mumford implies that a topologically regular complex algebraic surface in $\C^3$ with an isolated singular point is smooth. We prove that any Lipschitz regular complex algebraic set is smooth. No restriction on the dimension and no restriction on the singularity to be isolated is needed.
\end{abstract}

\maketitle

\section{Introduction}
 A classical theorem proved by David Mumford in 1961 \cite{M} is stated as follows: \emph{a normal complex algebraic surface with a simply connected link at a point $x$ is smooth in $x$}. This was a pioneer work in 
 the topology of singular algebraic surfaces. An important sequence of works in this area and some applications to the 3-folds topology were deeply stimulated by this theorem.  
 
  From a modern viewpoint this result can be seen as follows:   \emph{a topologically regular normal complex algebraic surface is smooth}. Since, in $\C^3$, a surface is normal if and only if it has isolated singularities, the result can be formulated as follows: \emph{topologically regular complex surface in $\C^3$ with isolated singularity is smooth}. The condition on the singularity to be isolated is important because there are examples of non-smooth and topologically regular surfaces, with non-isolated singularities.  
  
  There are also examples of non smooth surfaces in $\C^4$ (see Section 5 of \cite{T}) with topologically regular isolated singularity .
  
 Here we study Lipschitz regular complex algebraic sets. Lipschitz regularity simply means that the 
 complex algebraic set is a Lipschitz submanifold of $\C^n$. The main result is that Lipschitz regular points of a complex algebraic set are smooth points, i.e. the conclusion of the Mumford 
 theorem holds without any restrictions. 
 
 There are two basic ingredients in the proof. The first one is the theory of normal embedding, created by Birbrair and Mostowski \cite{BM} (The paper contains some lines describing the theory). Using the Bernig-Lytchak \cite{BL} map (see also \cite{BFN}), one can make a reduction of a Lipschitz regular point to its tangent cone. The second ingredient is the extraordinary theorem by David Prill \cite{P}. D. Prill proved that algebraic and topologically regular complex cones are smooth.

\bigskip

\noindent{\bf Acknowledgements}.  We would like to thank Andrei Gabrielov, Walter D. Neumann and Donu Arapura for their interest on this work.

\section{Normally embedded sets}

Given a closed and connected subanalytic subset $X\subset\R^m$ the
\emph{inner metric}  on $X$  is defined as follows: given two points $x_1,x_2\in X$, $d_X(x_1,x_2)$  is the infimum of the lengths of rectifiable paths on $X$ connecting $x_1$ to $x_2$. According to the real cusp $x^2=y^3$, one can see that the inner metric is not necessarily bi-Lipschitz equivalent to the Euclidean metric on $X$. A subanalytic set is called
\emph{normally embedded} if these two metrics (inner and Euclidean)
are bi-Lipschitz equivalent.

\begin{proposition}\label{proposition A}
Let $X\subset\C^n$ be an algebraic subset. Let $x_0\in X$ be such that the reduced tangent cone 
$T_{x_0}X:=|C_{x_0}X|$ is a linear subspace of $\C^n$. If there exists a neighborhood $U$ of $x_o$ in $X$ such  that $U$ is normally embedded, then $x_0$ is a smooth point of $X$.
\end{proposition}
\begin{proof} Since $T_{x_0}X$ is a linear subspace  of $\C^n$, we can consider the orthogonal projection
$$P\colon \C^n\rightarrow T_{x_0}X.$$
We may suppose that $x_0=0$ and $P(x_0)=0$. 

Notice that the germ of the restriction of the orthogonal projection $P$ to $X$ has the following properties:
\begin{enumerate}
\item It is finite complex analytic germ of map;
\item If $\gamma :[0,\varepsilon)\rightarrow X$ is a real analytic arc, such that $\gamma(0)=0$, then the  arcs 
$\gamma$ and $P\circ\gamma$ are tangent at $0$.
\end{enumerate}

Then the germ at $0$ of $P_{| X}\colon X\rightarrow T_{0}X$ is a ramified cover and the ramification locus is 
the germ of a codimension $1$ complex analytic subset $\Sigma$ of the plane $T_{0}X$. 

The multiplicity of $X$ at $0$ can be interpreted as the degree $d$ of this germ of ramified covering map, i.e. 
there are open neighborhoods $U_1$ of $0$ in $X$ and $U_2$ of $0$ in $T_{0}X$, such that $d$ is the
degree of the topological covering:
$$P_{| X}\colon X\cap U_1\setminus P_{| X}^{-1}(\Sigma)\rightarrow T_{0}X\cap U_2\setminus \Sigma.$$ 
Let us suppose that the degree $d$ is bigger than 1. Since $\Sigma$ is a codimension $1$ complex analytic subset of the space $T_{0}X$, there exists a unit tangent vector $v_0\in T_{0}X\setminus C_0\Sigma$.
where $C_0\Sigma $ is the tangent cone of $\Sigma$ at $0$. 

Since $v_0$ is not tangent to $\Sigma$ at $0$, there exists a positive real number $k$ such that the real cone $$\{v\in T_{0}X \ | \ |v-tv_0|< tk \ \forall \  0<t<1\}$$ does not intersect the set $\Sigma$. Since we have assumed that the degree $d\geq 2$, we have two (different) liftings $\gamma_1(t)$ and $\gamma_2(t)$ of the half-line $r(t)=tv_0$,  i.e. $P(\gamma_1(t))=P(\gamma_2(t))=tv_0$ and they satisfy $\|\gamma_i(t)-0\|=t, \ i=1,2.$   Since $P$ is the orthogonal projection on the reduced tangent cone $T_{0}X $, the vector $v_0$ is the unit tangent vector to the images 
$P\circ \gamma_1$ and $P\circ \gamma_2$ of the arcs $\gamma_1$ and $\gamma_2$ at $0$. By
construction, we have $\rm{dist}(\gamma_i(t),P_{| X}^{-1}(\Sigma))\geq kt$ ($i=1,2$).

On the other hand, any path in $X$ connecting $\gamma_1(t)$ to $\gamma_2(t)$ is the lifting of a loop, based at the point $tv_0$, which is not contractile in the germ, at $0$, of  $T_{0}X\setminus {\Sigma}$. Thus such a path must be longer than $kt$. It implies that the inner distance, $\rm{d_{inner}}(\gamma_1(t),\gamma_2(t))$, in $X$, between $\gamma_1(t)$ and $\gamma_2(t)$, is at least $kt$. But, since $\gamma_1(t)$ and $\gamma_2(t)$ are tangent at $0$, that is $\displaystyle\frac{|\gamma_1(t)-\gamma_2(t)|}{t}\to 0 \ \mbox{as} \ t\to 0^+,$ and $k>0$, we get a contradiction to  the fact that $X$ is normally embedded near $0$.
\end{proof}

\section{Main result}

All the sets considered in the paper are supposed to be equipped with the Euclidean metric. When we consider the inner metric, then we will emphasize it clearly.

\begin{definition}{\rm A subanalytic subset $X\subset\R^n$ is called \emph{Lipschitz regular at} $x_0\in X$ if there is an open neighborhood $U$ of $x_0$ in $X$ which is subanalytically bi-Lipschitz homeomorphic to an Euclidean ball.}
\end{definition}

\begin{theorem} Let $X\subset\C^n$ be a complex algebraic variety. If $X$ is Lipschitz regular at $x_0\in X$, then $x_0$ is a smooth point of $X$.
\end{theorem}

\begin{proof} Let $X\subset\C^n$ be a complex algebraic variety. We start the proof from the following lemma:

\begin{lemma}\label{neighborhood} If $X$ is Lipschitz regular at $x_0\in X$, then there exists a neighborhood $U$ of $x_0$ in $X$  which is  normally embedded in $\C^n$.
\end{lemma}

\begin{proof}[Proof of the Lemma] Let $X$ be Lipschitz regular at $x_0$. By definition, there exists a neighborhood $U$ of $x_0$ in $X$ and a subanalytic bi-Lipschitz homeomorphism $\psi : B \rightarrow U$, where $B$ is an Euclidean ball. Since $\psi$ induces a bi-Lipschitz homeomorphism between the metric spaces $(B,d_B)$ ($B$ equipped with its inner metric) and $(U,d_U)$ ($U$ equipped with its inner metric) (see \cite{BM}), we have positive real constants $\lambda_1$ and $\lambda_2$ such that 
$$ \lambda_1|p-q|\leq |\psi(p)-\psi(q)|\leq \lambda_2|p-q| \ \forall p,q\in B$$ and $$\lambda_1d_B(p,q)\leq d_U(\psi(p),\psi(q))\leq \lambda_2 d_B(p,q) \ \forall p,q\in B.$$ On the other hand,  Euclidean balls are normally embedded; or more precisely, 
$d_B(p,q)=|p-q|$ for all $p,q\in B$. By inequalities above, we have that 
$$d_U(\psi(p),\psi(q))\leq\frac{\lambda_2}{\lambda_1}|\psi(p)-\psi(q)| \ \forall \ p,q\in B.$$ In other words, $U$ is normally embedded.
\end{proof}

An algebraic subset of $\C^n$ is called a \emph{complex cone} if it is a union of one-dimensional linear subspaces of $\C^n$. The next result was proved by David Prill in \cite{P}.

\begin{lemma}[Theorem of Prill \cite{P}] \label{prill} Let $V\subset \C^n$ be a complex cone. If $0\in V$ has a neighborhood homeomorphic to a Euclidean ball, then $V$ is a linear subspace of $\C^n$
\end{lemma} 

Let us suppose that $X$ is Lipschitz regular at $x_0$. Let $C_{x_0}X$ be the tangent cone
of $X$ at $x_0$. Let $h\colon U\rightarrow B$ be 
a subanalytic bi-Lipschitz homeomorphism where $U$ is a neighborhood of $x_0$ in 
$X$ and $B\subset\R^N$ is an Euclidean ball centered at origin $0\in\R^N$. Let us suppose that $h(x_0)=0$.  
Then the derivative of $h$, 
$dh\colon |C_{x_0}X| \rightarrow T_0B,$
 defined by Bernig and Lytchak \cite{BL} (see also \cite{BFN}), is a bi-Lipschitz homeomorphism between the reduced tangent cones  $|C_{x_0}X|$  and $T_0B=\R^N$. In particular, it proves that $|C_{x_0}X|$ is also bi-Lipschitz regular at $x_0$. By the Theorem of Prill (Lemma \ref{prill}) , 
 we have that the reduced cone $|C_{x_0}X|$ is a linear subspace of $\C^n$. In order to finish the proof we just need to recall Lemma \ref{neighborhood} and Proposition \ref{proposition A}.
\end{proof}

\end{document}